\documentclass[11pt]{amsart}
\usepackage{mathrsfs}
\usepackage{amssymb}
\usepackage[all]{xy}
\usepackage{amsthm}
\usepackage{array}
\usepackage{amsmath}

\newtheorem{theorem}{Theorem}[section]
\newtheorem{lemma}[theorem]{Lemma}
\newtheorem{proposition}[theorem]{Proposition}
\newtheorem{corollary}[theorem]{Corollary}

\newtheorem{definition}[theorem]{Definition}
\newtheorem{example}[theorem]{Example}

\def\be{\begin{equation}}
\def\ee{\end{equation}}
\def\br{\begin{eqnarray}}
\def\er{\end{eqnarray}}

 \title{Complex balanced spaces }
 \author{ Jixiang Fu}
 \author{Lingxu Meng}
 \author{Wei Xia }
 \address{Institute of Mathematics\\  Fudan University \\ Shanghai
200433, China}
\email{majxfu@fudan.edu.cn}
\email{11110180019@fudan.edu.cn}
\email{11110180005@fudan.edu.cn}
\thanks{Fu is supported in part by NFSC}
 \date{}

\begin{document}
\maketitle

\begin{abstract}
In this paper, the concept of balanced manifolds are generalized to reduced complex spaces:  the class $\mathscr{B}$ and balanced spaces. Compared with the case of K\"ahlerian, the class $\mathscr{B}$ is similar to the Fujiki class $\mathscr{C}$ and the balanced space is similar to the K\"ahler space.  Some properties about these complex spaces are obtained, and the relations between the balanced spaces and  the class $\mathscr{B}$ are studied.
\end{abstract}

\textbf{keywords}:  class $\mathscr{B}$,  balanced space, balanced metric, Fujiki class $\mathscr{C}$

\textbf{AMSC}: 32C15, 32C10, 53C55,

\section{Introduction}

\quad In complex geometry, complex manifolds which admit K\"ahler metrics are studied by many mathematicians. Its topology and geometry have been understood very deep. In 1978, A. Fujiki generalized first the concept ``K\"ahler" to general complex spaces: the K\"ahler space which is a complex space admitting a strictly positive closed $(1, 1)$-form and the class $\mathscr{C}$ consisting of the reduced compact complex spaces  which is the meoromorphic image of a compact K\"ahler space (we call it the Fujiki class). In \cite{16}, Fujiki proved that if a compact reduced complex space $X$ is K\"ahler, then every irreducible component of the Douady space $D_{X, red}$ is compact. On the other hand, he also proved that if $X$ is in the  Fujiki class $\mathscr{C}$, then every irreducible component of the Barlet space $B(X)$ is compact. In \cite{24} and \cite{25}, J. Varouchas proved that any reduced complex space in the Fujiki class $\mathscr{C}$ has a proper modification which is a compact K\"ahler manifold.

In non-K\"ahler geometry, the complex manifold with a balanced metric is very important. A \emph{balanced} metric on a complex $n$-dimensional manifold is a hermitian metric such that its associated hermitian form $\omega$ satisfies $d(\omega^{n-1})= 0$. In \cite{22}, Michelsohn observed that the existence of a balanced metric is equivalent to the existence of a closed strictly positive $(n-1, n-1)$-form $\Omega$, and hence characterized the existence of balanced metrics as follows:
 \begin{theorem}
 [\cite{22}, Theorem 4.7] Let $M$ be a compact complex manifold. Then $M$ is balanced if and only if there is no nonzero positive current of bidegree $(1, 1)$ on $M$ which is the $(1, 1)$-component of a boundary.
 \end{theorem}

In 1993, L. Alessandrini and G. Bassanelli proved the following important theorem, which is the foundation of this paper.
\begin{theorem}
 [\cite{7}, Corollary 5.7] Let $f:M\rightarrow N$ be a modification of  compact complex manifolds. Then $M$ is balanced if and only if $N$ is balanced.
 \end{theorem}

 In this paper, we give two kinds of generalization of balanced manifolds to reduced complex spaces: the class $\mathscr{B}$  and balanced spaces. In section $2$, we study the properties of the class $\mathscr{B}$, and give some examples which are in $\mathscr{B}$, but neither in $\mathscr{C}$ nor balanced manifolds. In section $3$, we study families of reduced complex spaces over nonsingular curves and give some theorems on the total space being in $\mathscr{B}$. In section $4$, we study the relation between compact balanced spaces and the class $\mathscr{B}$. In the end, for the reader's convenience, we give an appendix about positivity.

\section{The class $\mathscr{B}$}

\begin{definition}
 A reduced  compact complex space $X$ is called in the class $\mathscr{B}$, if it has a resolution of singularities $\widetilde{X}$ which is a balanced manifold.
\end{definition}

Note that if a desingularization of $X$ is balanced, then every desingularization of $X$ is balanced. Indeed, if $X_1\to{X}$ and $X_2\to X$ are two resolutions of singularities of $X$, then there exists a bimeromorphic map  $f: X_1\dashrightarrow X_2$. Let $\Gamma\subseteq X_1\times X_2$ be the graph of $f$, and $p_1:\Gamma\rightarrow X_1$, $p_2:\Gamma\to X_2$ the two projections on $X_1$, $X_2$, respectively. Then $p_1$, $p_2$ are modifications. If $\widetilde{\Gamma}$ is a resolution of singularities of $\Gamma$, then $\widetilde{\Gamma}\to X_1$ and $\widetilde{\Gamma}\to X_2$ are modifications of compact complex manifolds. By Theorem 1.2, we know that $X_1$ is  balanced if and only if $\widetilde{\Gamma}$ is balanced, and then if and only if $X_2$ is balanced. Hence Definition 2.1 is not dependent on the choice of the resolution of singularities of $X$.

Using the same method as above, we can prove the following proposition.

\begin{proposition}
 The class $\mathscr{B}$ is invariant under bimeromorphic maps.
\end{proposition}

According to Definition 2.1 and Proposition 2.2, compact balanced manifolds, reduced compact complex spaces in the Fujiki class $\mathscr{C}$ and the normalizations of complex spaces in $\mathscr{B}$ are all in $\mathscr{B}$.
Obviously, if $X\in \mathscr{B}$ is nonsingular, then $X$ is a balanced manifold.

\begin{proposition}
If $X$ and $Y$ are reduced compact complex spaces, then $X\times Y$ is in the class $\mathscr{B}$ if and only if   $X$ and $Y$ are both in the class $\mathscr{B}$.
\begin{proof}
Let $f:(\widetilde{X}, \widetilde{E})\rightarrow(X, E)$ and $g:(\widetilde{Y}, \widetilde{F})\rightarrow (Y, F)$ be resolutions of singularities, where $\widetilde{E}$ and  $\widetilde{F}$ are the exceptional sets, and  $f(\widetilde{E})= E$, $g(\widetilde{F})= F$. If we define
\begin{displaymath}
G:=(E\times Y)\cup(X\times F),  \qquad  \widetilde{G}:=(\widetilde{E}\times \widetilde{Y})\cup(\widetilde{X}\times \widetilde{F}),
\end{displaymath}
then
\begin{displaymath}
f\times g:(\widetilde{X}\times \widetilde{Y}, \widetilde{G})\rightarrow (X\times Y, G)
\end{displaymath}
is a resolution of singularities of $X\times Y$.

Now, if $X$, $Y\in\mathscr{B}$, then $\widetilde{X}$ and $\widetilde{Y}$ are balanced manifolds, so $\widetilde{X}\times\widetilde{Y}$ is balanced. Hence $X\times Y \in\mathscr{B}$. Conversely, if $X\times Y\in\mathscr{B b 2}$, then $\widetilde{X}\times\widetilde{Y}$ is balanced by \cite{22}, Proposition 1.9 (i). By \cite{22}, Proposition 1.9 (ii) , $\widetilde{X}$ and $\widetilde{Y}$ are balanced, and hence, $X$ and $Y$ are both in the class $\mathscr{B}$.
\end{proof}
\end{proposition}

Using this proposition, we can construct some examples of complex spaces in $\mathscr{B}$ which are neither balanced manifolds nor in $\mathscr{C}$. If $X$ is a singular reduced compact complex space in $\mathscr{C}$ and $Y$ is a compact balanced manifold but not in $\mathscr{C}$, then $X\times Y$ is in $\mathscr{B}$, but it is neither a balanced manifold nor in $\mathscr{C}$. Obviously, $X\times Y$ is singular, so it is not a  balanced manifold. By Proposition 2.3, $X\times Y\in \mathscr{B}$. Consider the surjective holomorphic map $X\times Y\rightarrow Y$, by \cite{16}, Lemma 4.6.(3), we know $X\times Y$ is not in $\mathscr{C}$.

Suppose $dimX=n$. For $n=2$,  $X\in\mathscr{B}$ if and only if $X\in\mathscr{C}$, since for a smooth complex surface, the balanced and K\"ahlerian properties are equivalent.
For any $n\geq4$, if $Y$ is a singular projective variety of dimension $n-3$ and $I_3$ is the Iwasawa manifold, then $Y\times I_3$ is in class $\mathscr{B}$, which is neither a balanced manifold nor in $\mathscr{C}$.

If $X$ is a reduced compact complex space of pure dimension, then $X \in\mathscr{B}$ if and only if every irreducible component of $X$ is in $\mathscr{B}$. Indeed, if let  $\widetilde{X_1}$, $\ldots$, $\widetilde{X_r}$ be the resolutions of singularities of $X_1$, $\ldots$, $X_r$,  all the irreducible components of $X$, then the disjoint union $\widetilde{X}$:=$\widetilde{X_1} \amalg \ldots \amalg \widetilde{X_r}$ is a resolution of singularities of $X$. Hence the conclusion follows since $\widetilde{X}$ is balanced if and only if $\widetilde{X_1}$, $\ldots$, $\widetilde{X_r}$ are all balanced.

In the following we need the definition of a smooth morphism (\cite{11}, (0.4)). A surjective holomorphic map $f:X\rightarrow Y$ between reduced complex spaces is called a \emph{smooth} morphism, if for all $x\in X$, there is an open neighbourhood $W$ of $x$ in $X$, an open neighbourhood $U$ of $f(x)$ in $Y$, such that $f(W)= U$ and there is a commutative diagram
\begin{displaymath}
\xymatrix{
  W \ar[d]_{g} \ar[r]^{f|_W} &  U      \\
  \Delta^n\times U \ar[ur]_{pr_2}     }
\end{displaymath}
where $n= dimX- dimY$, $g$ is an isomorphism, $pr_2$ is the second projection, and $\Delta^n$ is a small polydisc. Moreover, if $dimX= dimY$, a smooth morphism is exactly a \emph{local isomorphism}.

Obviously, if $f:X\rightarrow Y$ is a smooth morphism, and $Y$ is a complex manifold, then $X$ must also be a complex manifold and $f$ is a submersion between complex manifolds.

\begin{proposition}
Let $f:X\rightarrow Y$ be a smooth morphism of  reduced compact complex spaces. If $X \in\mathscr{B}$, then $Y \in\mathscr{B}$.
\end{proposition}
\begin{proof}
 Suppose $p:\widetilde{Y}\rightarrow Y$ is a resolution of singularities. Consider the following Cartesian
\begin{equation}\label{Car}
\xymatrix{
\widetilde{X}:=X\times_Y \widetilde{Y} \ar[d]_{q} \ar[r]^-{\tilde{f}}  & \widetilde{Y} \ar[d]^{p}  \\
  X \ar[r]_{f}  & Y}
\end{equation}
We have the following two claims.
\vspace{2mm}

\noindent \textbf{Claim 1.} $\tilde{f}$ is a smooth morphism, i.e. a submersion of complex manifolds.
\begin{proof}
For arbitrary $\tilde{x}=(x, \tilde{y})\in\widetilde{X}$, such that $f(x)= p(\tilde{y})$, we choose an open neighbourhood W of $x$ in $X$, an open neighbourhood U of $f(x)$ in $Y$, such that $W\cong\Delta^n\times U$, and under this isomorphism, $f|_W$ is exactly the second projection $pr_2:\Delta^n\times U\rightarrow U$. Considering the neighbourhoods $W\times_U p^{-1}(U)$ of $\tilde{x}=(x, \tilde{y})$ and $p^{-1}(U)$ of $\tilde{y}$, we have
\begin{displaymath}
W\times_U p^{-1}(U)\cong(\Delta^n\times U)\times_U p^{-1}(U)\cong\Delta^n\times p^{-1}(U)
\end{displaymath}
 and under this isomorphism, the second projection $\Delta^n\times p^{-1}(U)\rightarrow p^{-1}(U)$ is exactly $\tilde{f}:W\times_U p^{-1}(U)\rightarrow p^{-1}(U)$. Hence, $\tilde{f}$ is a smooth morphism.
\end{proof}
\vspace{2mm}

\noindent \textbf{Claim 2.} $q$ is a modification, hence a resolution of singularities of $X$.

\begin{proof}We set $\widetilde{F}\subset \widetilde{Y} $, the exceptional set of $p$,  and $F= p(\widetilde{F})$. For arbitrary $x\in X$, we choose $W$ and $U$ as above, then
\begin{displaymath}
f^{-1}(F)\cap W\cong\Delta^n\times (U\cap F)
\end{displaymath}
is a nowhere dense subset of $W\cong\Delta^n\times U$, hence  $f^{-1}(F)$ is nowhere dense in $X$. Moreover,
\begin{displaymath}
q^{-1}(f^{-1}(F))\cap W\times_U p^{-1}(U)= W\times_U (p^{-1}(U)\cap \widetilde{F})\cong  \Delta^n\times(p^{-1}(U)\cap \widetilde{F})
\end{displaymath}
is a nowhere dense subset of $W\times_U p^{-1}(U)\cong\Delta^n\times p^{-1}(U)$, hence  $q^{-1}(f^{-1}(F))$ is nowhere dense in $\widetilde{X}$.

On the other hand,
\begin{displaymath}
q:\widetilde{X}-q^{-1}(f^{-1}(F))=(X-f^{-1}(F))\times_{(Y-F)} (\widetilde{Y}-\widetilde{F})\rightarrow X-f^{-1}(F)
\end{displaymath}
 is obviously an isomorphism, since $ \widetilde{Y}-\widetilde{F}\rightarrow Y-F$ is isomorphic.
\end{proof}

Now, we can prove this proposition. Since $X \in\mathscr{B}$, $\widetilde{X}$ is a balanced manifold, so is $\widetilde{Y}$ by \cite{22}, Proposition 1.9(ii), hence $Y$ is in the class $\mathscr{B}$.
\end{proof}

\begin{proposition}
If $f:X\rightarrow Y$ is a local isomorphism of reduced compact complex spaces, then $X\in\mathscr{B}$ if and only if $Y\in\mathscr{B}$. In particular, if $f$ is an unramified covering map, then $X\in\mathscr{B}$ if and only if $Y\in\mathscr{B}$.
\end{proposition}
\begin{proof}
 Suppose $p:\widetilde{Y}\rightarrow Y$ is a resolution of singularities. Consider the Cartesian diagram (\ref{Car}). We have shown that $\tilde{f}$ is a local isomorphism and $q$ is a modification by Claim 1 and Claim 2 in the proof of Proposition 2.4.

If $Y\in\mathscr{B}$, then $\widetilde{Y}$ is a balanced manifold. Suppose $\Omega_{\widetilde{Y}}$ is a closed strictly positive $(n-1, n-1)$-form on $\widetilde{Y}$, where $n=dim\widetilde{Y}$. For any $\tilde{x}\in\widetilde{X}$, there is an open neighbourhood $\widetilde{W}$ of $\tilde{x}$ in $\widetilde{X}$ and an open neighbourhood $\widetilde{U}$ of $\tilde{f}(\tilde{x})$ in $\widetilde{Y}$, such that $\widetilde{f}\mid_{\widetilde{W}}:\widetilde{W}\rightarrow \widetilde{U}$ is an isomorphism. Obviously,
$(\tilde{f}^*\Omega_{\widetilde{Y}})\mid_{\widetilde{W}} = (\tilde{f}\mid_{\widetilde{W}})^*(\Omega_{\widetilde{Y}}\mid_{\widetilde{U}})$ is strictly positive, hence $\tilde{f}^*\Omega_{\widetilde{Y}}$ is a closed strictly positive $(n-1, n-1)$-form on $\widetilde{X}$. Hence $\widetilde{X}$ is balanced and  $X\in\mathscr{B}$. Conversely, it is obvious by Proposition 2.4.
\end{proof}

Before considering the inverse of Proposition 2.4, we recall some definitions of the forms and currents on complex spaces, following \cite{20}.

Let $X$ be a reduced complex space, and $X_{reg}$, which is a complex manifold, be the set of nonsingular points in $X$.

Suppose that $X$ is  a subvariety of a complex manifold $M$. Let $I_X^{p, q}(M)= \{\alpha\in A^{p, q}(M) \mid i^*\alpha= 0\}$, where $i: X_{reg}\rightarrow M$ is the inclusion. Define $A^{p, q}(X):= A^{p, q}(M)/I_X^{p, q}(M)$. It can be shown that $A^{p, q}(X)$ is independent of the embedding of $X$ into $M$. Hence, for any complex space $X$, we can define $A^{p, q}(X)$ by using the local embeddings in $\mathbb{C}^N$. More precisely, we define a sheaf of germs $\mathcal{A}^{p, q}_X$ of $(p, q)$-forms on $X$ and $A^{p, q}(X)$ as the group of its global sections. Similarly, we can also define $A_c^{p, q}(X)$ ($(p, q)$-forms with compact supports), $A^k(X)$ and $A_c^k(X)$.

We can naturally define $\partial:A^{p, q}(X)\rightarrow A^{p+1, q}(X)$,  $\overline{\partial}:A^{p, q}(X)\rightarrow A^{p, q+1}(X)$ and $d:A^k(X)\rightarrow A^{k+1}(X)$.

If $f: X\rightarrow Y$ is a holomorphic map of reduced complex spaces, then we can naturally define $f^*:A^{p, q}(Y)\rightarrow A^{p, q}(X)$ such that $f^*$ commutes with $\partial$, $\overline{\partial}$, $d$.

When $X$ is a subvariety of a complex manifold $M$, we define the space of currents on $X$
\begin{displaymath}
\mathcal{D}'^{2n-r}(X)=\mathcal{D}_r'(X):= \{T\in \mathcal{D}_r'(M)\mid T(u)= 0, \forall u\in I_{X, c}^r(M)\},
\end{displaymath}
where $\mathcal{D}_r'(M)$ is the space of currents on $M$, $I_{X, c}^r(M)= \{\alpha\in A_c^{p, q}(M) \mid i^*\alpha= 0\}$. As the case of $A^r(X)$, we can define a space $\mathcal{D}_r'(X)$ of the currents on any reduced complex space $X$. We define
\begin{displaymath}
\mathcal{D}'^{n-p, n-q}(X)=\mathcal{D}_{p, q}'(X):= \{T\in \mathcal{D}_{p+q}'(X) \mid T(u)= 0, \forall u\in A_c^{r, s}(M), (r, s)\neq (p, q)\}.
\end{displaymath}
A current $T$ is called a \emph{$(p, q)$-current} on $X$, if $T\in\mathcal{D}'^{p, q}(X)$.
If $T\in \mathcal{D}_r'(X)=\mathcal{D}'^{2n-r}(X)$, we call $r$ the \emph{dimension},  $2n-r$ the \emph{degree} of the current $T$. If $T\in \mathcal{D}_{p, q}'(X)=\mathcal{D}'^{n-p, n-q}(X)$, we call $(p, q)$ the \emph{bidimension}, $(n-p, n-q)$ the \emph{bidegree} of the current $T$.

\textbf{Note.} We use the terminology here as most papers now, but in \cite{22}, \emph{$(p, q)$-current} denotes a current of bidimension $(p, q)$.

When $f: X\rightarrow Y$ is a holomorphic map of reduced compact complex spaces, we can define $f_*:\mathcal{D}_r'(X)\rightarrow \mathcal{D}_r'(Y)$ $f_*T(u):= T(f^*u)$ for any  $u\in A_c^r(Y)$. Obviously, $f_*(\mathcal{D}_{p, q}'(X))\subseteq \mathcal{D}_{p, q}'(Y)$.

A current $T\in \mathcal{D}'^{p, p}(X)$ is called \emph{real} if for every $\alpha\in A_c^{2n-2p}(X)$, $T(\overline{\alpha})= \overline{T(\alpha)}$. A real current $T\in \mathcal{D}'^{p, p}(X)$ is called \emph{the $(p, p)$-component of a boundary} if there exists a current $S$ such that for any $\alpha\in A_c^{2n-2p}(X)$, $T(\alpha)= S(d\alpha^{n-p, n-p})$, where $\alpha^{n-p,n-p}$  is the $(n-p, n-p)$-part of $\alpha$. We write $T= \pi_{p, p}dS$, which is equivalent to that there exists a $(p, p-1)$-current Q such that $T=\partial\overline{Q}+ \overline{\partial} Q$.

A real $(p, p)$-form $\omega$ on $X$ is called \emph{strictly positive}, if there exist an open covering $\mathcal{U}= \{U_\alpha\}$ of $X$ with an embedding $i_\alpha:U_\alpha\rightarrow V_\alpha$ of $U_\alpha$ into a subdomain $V_\alpha$ in $\mathbb{C}^{n_\alpha}$ and a strictly positive $(p, p)$-form $\omega_\alpha$ on $V_\alpha$, such that $\omega\mid_{U_\alpha}= i_\alpha^*\omega_\alpha$, for each $\alpha$.

Now, we consider the inverse of Proposition 2.4, which is a generalization of Corollary 3.10 in \cite{1}.

\begin{proposition} Let $f:X\rightarrow Y$ be a smooth morphism of  reduced compact complex spaces, and $n= dimX> m=dim Y\geq 2$. If $Y \in\mathscr{B}$ and there exists a point $y_0$ in $Y$ such that the fibre $f^{-1}(y_0)$ is not the $(m, m)$-component of a boundary, then $X \in\mathscr{B}$.
\end{proposition}
\begin{proof}
Considering the Catesian diagram (\ref{Car}), we have shown that $\tilde{f}$ is a submersion of complex manifolds and $q$ is a resolution of singularities of $X$.

 We first claim that for every $\tilde{y}\in p^{-1}(y_0)$, $\tilde{f}^{-1}(\tilde{y})$ can not be written as $\partial\overline{Q}+\overline{\partial}Q$ for some current $Q$ of bidegree $(m, m-1)$ on $\widetilde{X}$. If not, since $\tilde{f}^{-1}(\tilde{y})= f^{-1}(y_0)\times \{\tilde{y}\}$,  we have
\begin{displaymath}
[f^{-1}(y_0)]= q_*[\tilde{f}^{-1}(\tilde{y})]=q_*(\partial\overline{Q}+\overline{\partial}Q)=\partial\overline{q_*Q}+\overline{\partial}q_*Q,
\end{displaymath}
which contradicts the assumption.

Now suppose $\tilde{y}'$ is any point in $\widetilde{Y}$. Since $\tilde{f}$ is smooth and $[\tilde{y}]=\lambda[\tilde{y}']$ in $H^{2m}(\widetilde{Y} , \mathbb{R})$ for some constant $\lambda\in\mathbb{R}$, then
\begin{displaymath}
\lambda[\tilde{f}^{-1}(\tilde{y}')]= \lambda\tilde{f}^*[\tilde{y}']= \tilde{f}^*[\tilde{y}]=[\tilde{f}^{-1}(\tilde{y})]
\end{displaymath}
in $H^{2m}(\widetilde{X}, \mathbb{R})$, where $\tilde{y}\in p^{-1}(y_0)$, and $\tilde{f}^*: H^{2m}(\widetilde{Y}, \mathbb{R})\rightarrow H^{2m}(\widetilde{X}, \mathbb{R})$ is the pull back of $\tilde{f}$. Hence for every $\tilde{y}'\in \widetilde{Y}$, $\tilde{f}^{-1}(\tilde{y}')$ is not the $(m, m)$-component of a boundary. By \cite{1}, Corollary 3.10, $\widetilde{X}$ is balanced, hence $X \in\mathscr{B}$.
\end{proof}

Using the similar method, we can easily get a proposition of this type about the Fujiki class $\mathscr{C}$, which is a generalization of Corollary 3.8 in \cite{1}.
\begin{proposition} Let $f:X\rightarrow Y$ be a smooth morphism of reduced compact complex spaces. Let $n= dimX$ and  $m=dim Y\geq 2$ satisfy $n=m+1$. If $Y \in \mathscr{C}$ and there exists a point $y_0$ in $Y$ such that the fibre $f^{-1}(y_0)$ is not the $(m, m)$-component of a boundary, then $X \in\mathscr{C}$.
\end{proposition}

\section{Families  of complex spaces over a nonsingular curve}

\quad In this section, we study families of complex spaces over a curve. It should be useful in the study of deformations and moduli spaces of complex spaces. The following two definitions are generalizations of the corresponding notions defined in \cite{22}.

\begin{definition} Let $X$ be a reduced compact complex space of pure dimension $n$, and $f:X\rightarrow C$ a holomorphic map onto a nonsingular compact complex curve.

(1) $f$ is called \textup{essential}, if for every $p\in C$, no linear combination $\sum_jc_j[F_j]$ can be the $(1, 1)$-component of a boundary on $X$, where the ${F_j}^,s$ are all the irreducible components of the fibre $f^{-1}(p)$, $c_j\geq 0$ and at least one of the ${c_j}^,s$ is positive;

(2) $f$ is called \textup{topologically essential}, if for every $p\in C$, no linear combination $\sum_jc_j[F_j]$ is zero in $H_{2n-2}(X, \mathbb{R})$, where the ${F_j}^,s$ are all the irreducible components of the fibre $f^{-1}(p)$, $c_j\geq 0$ and at least one of the ${c_j}^,s$ is positive.
\end{definition}

We first note that, for any  reduced compact complex space $X$ of pure dimension $n$ and the holomorphic map $f:X\rightarrow C$ onto a nonsingular compact complex curve, by the open mapping theorem (\cite{19}, page 109), $f$ is an open map. Hence for every $p\in C$, every irreducible component of  $f^{-1}(p)$ has dimension $n-1$ (\cite{14}, 3.10 Theorem).

Moreover, the``essential" hypothesis implies ``topologically essential", and when $X$ is a complex manifold, the two concepts are equivalent in the cases when $dimX=2$ or $H^2(\Omega^\bullet)=0$, where $H^2(\Omega^\bullet)$ is the second cohomology of the complex of  global holomorphic forms
\begin{displaymath}
0\rightarrow \Gamma(X,\mathcal O)\rightarrow\Gamma(X,\Omega^1)\rightarrow\Gamma(X,\Omega^2)\rightarrow\Gamma(X,\Omega^3)\rightarrow\cdots.
\end{displaymath}
where the differentials in the complex are $\partial$. For more details, see \cite{22}, Proposition 5.3.

Now, we can generalize \cite{22}, Theorem 5.5 to the following situation.

\begin{theorem} Suppose $X$ is a purely $n$-dimensional compact normal complex space which admits an essential holomorphic map $f:X\rightarrow C$ onto a nonsingular compact complex curve $C$, and $X$ has a resolution of singularities $\pi:\widetilde{X}\to X$ such that no nonzero nonnegative linear combination of hypersurfaces contained in the exceptional set of $\pi$ can be the $(1, 1)$-component of a boundary on $\widetilde{X}$.  If every nonsingular  fiber of $f$ is a balanced manifold, then $X\in \mathscr{B}$.
\end{theorem}
\begin{proof}
Set $\tilde{f}:=f\circ\pi$. For every $p\in C$, set $f^{-1}(p)=\cup_iV_i$, where $V_i$ are all the irreducible components of $f^{-1}(p)$, which have dimension $n-1$. Since $X$ is normal, $codimX_s \geq 2$, where $X_{s}$ is the set of singular points of $X$. So
\begin{displaymath}
\pi^{-1}(V_i)= \widetilde{V_i}\cup\cup_j E_{ij},
\end{displaymath}
where $\widetilde{V_i}=\overline{\pi^{-1}(V_i-X_s)}$ is the strict transform of $V_i$, and $E_{ij}$ are all irreducible components of $\pi^{-1}(V_i)$ contained in the exceptional set of $\pi$. It is possible that some $E_{i j}$ are contained in other $E_{k l}$ or $\widetilde{V_k}$. We denote any $E_{i j}$, which is not properly contained in other $E_{k l}$ or $\widetilde{V_k}$, by $E_{i j'}$  and we denote any $E_{i j}$, which is properly contained in other $E_{k l}$ or $\widetilde{V_k}$, by $E_{i j''}$ (i.e. there exists other $E_{k l}$ or $\widetilde{V_k}$, such that $E_{i j''}\subsetneqq E_{k l}$ or $\widetilde{V_k}$), then
\begin{displaymath}
\tilde{f}^{-1}(p)=\cup_i (\widetilde{V_i}\cup \cup_{j'} E_{i j'})
\end{displaymath}
and $codimE_{i j'}=1$.

We need the following two claims.
\vspace{2mm}

\noindent \textbf{Claim 1.} $\tilde{f}$ is essential.
\begin{proof}
If not, we have
\begin{displaymath}
\Sigma_ia_i[\widetilde{V_i}] + \Sigma_{ij'}b_{ij'}[E_{ij'}] = \pi_{n-1,n-1}dT,
\end{displaymath}
for some current $T$ on $\widetilde{X}$, $a_i$, $b_{ij'}$ $\geq 0$ and at least one of the ${a_i}^,s$, ${b_{ij'}}^,s$ is positive. Since $\pi(E_{ij'}) \subseteq X_s$ has codimension $\geq 2$, $\pi_*[E_{ij'}] = 0$. Obviously, $\pi_*[\widetilde{V_i}] = [V_i]$,  we have
\begin{displaymath}
\Sigma_ia_i[V_i]=\pi_{n-1,n-1}d(\pi_{*}T)
\end{displaymath}
through $\pi_*$. Since $f$ is essential, $a_i=0$ for all $i$. So
\begin{displaymath}
\Sigma_{ij'}b_{ij'}[E_{ij'}] = \pi_{n-1,n-1}dT,
\end{displaymath}
where $b_{ij'}$ $\geq 0$ and at least one of the ${b_{ij'}}^,s$ is positive. It contradicts the assumption on $\widetilde{X}$.
\end{proof}
\vspace{2mm}

\noindent \textbf{Claim 2.} For every $p\in C$, if $\tilde{f}^{-1}(p)$ is nonsingular, then it is balanced.
\begin{proof}
Since $\tilde{f}^{-1}(p)=\cup_i(\widetilde{V_i}\cup\cup_{j'} E_{ij'})$ is nonsingular, we have
\begin{displaymath}
\widetilde{V_i}\cap \widetilde{V_k} = \emptyset, \quad \forall i\neq k;
\end{displaymath}
\begin{displaymath}
\widetilde{V_i}\cap E_{k l'}= \emptyset, \quad \forall i, k, l'.
\end{displaymath}
Obviously, the sets for different points $p\in C$, $\{V_i\}_i$ don't intersect one another. Since for any $i, j$, $E_{ij}$ is contained in some $E_{k l'}$ or $\widetilde{V_k}$, we have $\widetilde{V_i}\cap E_{ij}=\emptyset$.
On the other hand, if $V_i\cap X_s\neq\emptyset$, then the intersection of $\widetilde{V_i}$ and $\cup_jE_{ij}$ is not empty, so for all $i$, $V_i$ don't intersect with $X_s$. Hence, the map
\begin{displaymath}
\pi:\tilde{f}^{-1}(p)=\sqcup_i\widetilde{V_i}\rightarrow f^{-1}(p)=\sqcup_i V_i
\end{displaymath}
is an isomorphism. Since every nonsingular  fiber of $f$ is balanced and $\tilde{f}^{-1}(p)$ is nonsingular, $\tilde{f}^{-1}(p)$ is balanced.
\end{proof}

Now, by the Claim 1 and 2, $\widetilde{X}$ is balanced according to \cite{22}, Theorem 5.5. Hence, $X$ $\in$ $\mathscr{B}$.
\end{proof}

By the above theorem, we have the following corollary immediately.

\begin{corollary} Suppose $X$ is a purely $n$-dimensional compact normal complex space which admits an essential holomorphic map $f:X\rightarrow C$ onto a nonsingular compact complex curve $C$, and $X$ has a desingularization  $\widetilde{X}$ whose exceptional set has codimension $\geq 2$.  If every nonsingular  fiber of $f$ is a balanced manifold, then $X\in \mathscr{B}$.
\end{corollary}

\begin{theorem}
Let $X$ be a normal compact complex surface which admits a topologically essential holomorphic map onto a nonsingular compact complex curve. If $X$ has a resolution of singularities $\pi:\widetilde{X}\to X$, such that no nonzero nonnegative linear combination of hypersurfaces contained in the exceptional set of $\pi$, is zero in $H_2(\widetilde{X}, \mathbb{R})$, then $X$ is in the Fujiki class $\mathscr{C}$.
\begin{proof}
 Suppose $f:X\rightarrow C$ is the topologically essential holomorphic map onto a nonsingular compact complex curve $C$. Set $\tilde{f}:=f\circ\pi$.

We claim that $\tilde{f}$ is topologically essential. If not, there exists a point $p\in C$,
\begin{displaymath}
f^{-1}(p)=\cup_iV_i,
\end{displaymath}
\begin{displaymath}
\tilde{f}^{-1}(p)=\cup_i(\widetilde{V_i}\cup\cup_{j'} E_{ij'}),
\end{displaymath}
where $V_i$, $\widetilde{V_i}$, $E_{ij'}$, and $X_s$ are defined as in the proof of Theorem 3.2, and $a_i$, $b_{ij'}$ $\geq 0$ and at least one of the ${a_i}^,s$, ${b_{ij'}}^,s$ is positive, such that
\begin{displaymath}
\Sigma_ia_i[\widetilde{V_i}]+ \Sigma_{ij'}b_{ij'}[E_{ij'}]=0
\end{displaymath}
in $H_2(\widetilde{X}, \mathbb{R})$. Since $\pi(E_{ij'})\subseteq X_s$ has codimension $\geq 2$, $\pi_*[E_{ij'}]=0$.
Obviously, $\pi_*[\widetilde{V_i}] = [V_i]$,  we have
\begin{displaymath}
\Sigma_ia_i[V_i] = 0
\end{displaymath}
through $\pi_*$ in $H_2(X, \mathbb{R})$. Since $f$ is topologically essential, $a_i=0$ for all $i$. So
\begin{displaymath}
\Sigma_{ij'}b_{ij'}[E_{ij'}] = 0
\end{displaymath}
where $b_{ij'}$ $\geq 0$ and at least one of the ${b_{ij'}}^,s$ is positive. It contradicts the assumption on $\widetilde{X}$.

By \cite{22}, Corollary 5.8, $\widetilde{X}$ is K\"ahler, so $X$ is in the Fujiki class $\mathscr{C}$.
\end{proof}
\end{theorem}

\begin{corollary}
Let $X$ be a normal compact complex surface which admits a topologically essential holomorphic map onto a nonsingular compact complex curve. If the Betti number $b_3(X)= 0$, then $X$ is in the Fujiki class $\mathscr{C}$.
\begin{proof}
Using Theorem 3.4 and following lemma, we get this corollary immediately.
\end{proof}
\end{corollary}

\begin{lemma}
Let $f: X\rightarrow Y$ be a modification between reduced compact complex spaces of dimension $n$. If $Y$ is normal and the Betti number $b_{2n-1}(Y)= 0$, then there is a exact sequence
\begin{displaymath}
\xymatrix{
0\ar[r] &H_{2n-2}(E, \mathbb{R})\ar[r]^{i_*} &H_{2n-2}(X, \mathbb{R})\ar[r]^{f_*} &H_{2n-2}(Y, \mathbb{R})
}
\end{displaymath}
where $E$ is the exceptional set of $f$, $i: E\rightarrow X$ is the inclusion. Moreover, $H_{2n-2}(E, \mathbb{R})= \oplus_j \mathbb{R}[E_j]$, where $\{E_j\}_j$ are all the $(n-1)$-dimensional irreducible components of $E$.
\end{lemma}
\begin{proof}
Consider the following commutative diagram of exact sequences of Borel-Moore homology
\begin{displaymath}
\xymatrix{
H_{2n-1}(X, \mathbb{R})\ar[d]^{f_*} \ar[r]& H^{BM}_{2n-1}(U, \mathbb{R}) \ar[d]^{\cong} \ar[r]& H_{2n-2}(E, \mathbb{R}) \ar[d]\ar[r]^{i_*}& H_{2n-2}(X, \mathbb{R})\ar[d]^{f_*}\ar[r]&H^{BM}_{2n-2}(U, \mathbb{R})\ar[d]^{\cong}\\
H_{2n-1}(Y, \mathbb{R})       \ar[r]& H^{BM}_{2n-1}(V, \mathbb{R})     \ar[r] & H_{2n-2}(F, \mathbb{R})     \ar[r]& H_{2n-2}(Y, \mathbb{R})         \ar[r]&H^{BM}_{2n-2}(V, \mathbb{R}) }
\end{displaymath}
where $F= f(E)$, $U= X-E$, $V= Y- F$. Since $Y$ is normal, by \cite{19}, page 215, $codim_YF\geq 2$, so $H_{2n-2}(F, \mathbb{R})= 0$. Since $H_{2n-1}(Y, \mathbb{R})=0$, we obtain $H^{BM}_{2n-1}(V, \mathbb{R})=0$ by the second long exact sequence, hence $H^{BM}_{2n-1}(U, \mathbb{R})=0$, so $i_*$ is injective. If $\alpha\in Ker(f_*:H_{2n-2}(X, \mathbb{R})\rightarrow H_{2n-2}(Y, \mathbb{R}))$, i.e. $f_*(\alpha)=0$, then the image of $\alpha$ in $H^{BM}_{2n-2}(U, \mathbb{R})\cong H^{BM}_{2n-2}(V, \mathbb{R})$ is zero, hence $\alpha$ is in the image of $i_*$ by the first long exact sequence. Since $H_{2n-2}(F, \mathbb{R})= 0$, so $f_*i_*=0$, hence we get the short exact sequence in the lemma.

Let $E_1,...,E_r$ be  all the irreducible components of $E$ which have dimension $(n-1)$, $R_1,..., R_s$ all the  irreducible components of $E$ which have dimension smaller than $(n-1)$, $A=\cup_{i\neq j}(E_i\cap E_j)\cup \cup_{k\neq l}(R_k\cap R_l)\cup\cup_{i, k}(E_i\cap R_k)$, $E'_i=E_i-A\cap E_i$ and $R'_k=R_k-A\cap R_k$. So $E'_i$, $R'_k$, $i=1,...,r, k=1,...,s$ don't intersect one another and $E-A=\cup_i E'_i\cup\cup_k R'_k$. Now, we consider the exact sequence of Borel-Moore homology for $(E, A)$
\begin{displaymath}
\xymatrix{
H_{2n-2}(A, \mathbb{R})\ar[r] &H_{2n-2}(E, \mathbb{R})\ar[r] &H^{BM}_{2n-2}(E-A, \mathbb{R})\ar[r] &H_{2n-3}(A, \mathbb{R})
}
\end{displaymath}
By definition of $A$, we know $dimA\leq n-2$, so $H_{2n-2}(A, \mathbb{R})=H_{2n-3}(A, \mathbb{R})=0$. Hence $H_{2n-2}(E, \mathbb{R})= H^{BM}_{2n-2}(E-A, \mathbb{R})= \oplus_i H^{BM}_{2n-2}(E'_i, \mathbb{R})\oplus\oplus_k H^{BM}_{2n-2}(R'_k, \mathbb{R})$. By the similar method, consider the long exact sequences of Borel-Moore homology for $(E_i, A\cap E_i)$, $(R_i, A\cap R_i)$, respectively, and $dimR_k\leq n-2$, we obtain $H^{BM}_{2n-2}(E'_i, \mathbb{R})= H_{2n-2}(E_i, \mathbb{R})= \mathbb{R}[E_i]$ and $H^{BM}_{2n-2}(R'_k, \mathbb{R})= H_{2n-2}(R_k, \mathbb{R})= 0$. Hence, we get $H_{2n-2}(E, \mathbb{R})= \oplus_i \mathbb{R}[E_i]$.
\end{proof}

\section{Balanced spaces}

\quad Recall that a complex manifold $M$ is called \emph{$p$-K\"ahler}, if it admits a closed strictly positive $(p, p)$-form (see appendix).

\begin{definition}
 An $n$-dimensional reduced complex space $X$ is called balanced, if there is a closed strictly positive $(n-1, n-1)$-form $\Omega$ on $X$.
 \end{definition}

Obviously, balanced manifolds and K\"ahler spaces are balanced spaces. When $dimX= 2$, a compact balanced space $X$ is a compact K\"ahler space, hence in the class $\mathscr{B}$.

\begin{proposition}
 If $M$ is a $p$-K\"ahler manifold, then any $(p+1)$-dimensional analytic subspace $X$ of $M$ is a balanced space.
  \end{proposition}
\begin{proof}
Let $\Omega$ be a closed strictly positive $(p, p)$-form on $M$. Then, obviously, the restriction of $\Omega$ to $X$ is also a closed strictly positive $(p, p)$-form, so $X$ is a balanced space.
\end{proof}

Now we construct a singular compact balanced space, which is not a K\"ahler space (even not in the Fujiki class $\mathscr{C}$).

\begin{example}
[\cite{2}]Suppose
\begin{displaymath}
G=\{A \in GL(n+2, \mathbb{C})\mid A \mathbf{}=
\left(\begin{array}{ccc}
 1 & X & z   \\
 0 & I_n & Y   \\
 0 & 0 & 1
 \end{array} \right),
\forall z\in \mathbb{C}, X, Y^T\in \mathbb{C}^n            \}
\end{displaymath}
and $\Gamma\subseteq G$ consists of matrices with entries in $\mathbb{Z}[i]$. Define $\eta\beta_{2n+1}:= G/\Gamma$.

Since there exists a closed emmbedding $I_3\hookrightarrow \eta\beta_{2n+1}$
\begin{displaymath}
\mathbf{}
\left(\begin{array}{ccc}
 1 & x & z   \\
 0 & 1 & y   \\
 0 & 0 & 1
 \end{array} \right)
 \mapsto
 \mathbf{}
\left(\begin{array}{ccc}
 1 & X & z   \\
 0 & I & Y   \\
 0 & 0 & 1
 \end{array} \right)
 \end{displaymath}
 where $X= (x, 0, ... ,0)$, $Y= (y, 0, ... ,0)^T$, the Iwasawa manifold $I_3$ is a complex submanifold of $\eta\beta_{2n+1}$. Since any closed complex subspace of a compact complex space in the Fujiki class $\mathscr{C}$ is also in $\mathscr{C}$ (\cite{16}, Lemma 4.6 (3)) and the Iwasawa manifold $I_3$ is not in $\mathscr{C}$, so $\eta\beta_{2n+1}$ is not in $\mathscr{C}$.

If $H$ is a singular hypersurface in $\mathbb{CP}^n$, then $\eta\beta_5\times H$ is a singular compact balanced space, but not in the Fujiki class $\mathscr{C}$. Indeed, by \cite{3}, (4.7), $\eta\beta_5\times \mathbb{CP}^n$ is $(n+3)$-K\"ahler, so $\eta\beta_5\times H$ is a singular compact balanced space by Proposition 4.2.
\end{example}

Next we recall  several notations and theorems .

Let $M$ be a compact complex manifold. Define the  \emph{Bott-Chern cohomology group} of degree $(p, q)$
\begin{displaymath}
H^{p,q}_{BC}(M):= \frac{Ker(d: A^{p, q}(M)\rightarrow A^{p+q+1}(M))}{\partial\overline{\partial}A^{p-1, q-1}(M)}.
\end{displaymath}

and the \emph{Aeplli cohomology group} of degree $(p, q)$

\begin{displaymath}
H^{p,q}_{A}(M):= \frac{Ker(\partial\overline{\partial}: A^{p, q}(M)\rightarrow A^{p+1, q+1}(M))}{\partial A^{p-1, q}(M) + \overline{\partial}A^{p, q-1}(M)}.
\end{displaymath}
It is well known that all these groups can also be defined  by means of currents of corresponding degree.

Moreover, for every $(p, q)\in \mathbb{N}^2$, the identity induces a natural map
\begin{displaymath}
i: H^{p,q}_{BC}(M)\rightarrow H^{p,q}_{A}(M).
\end{displaymath}
In general, the map $i$ is neither injective nor surjective. When $M$ is a compact complex manifold satisfying $\partial\overline{\partial}$-lemma, for every $(p, q)\in \mathbb{N}^2$, $i$ is an isomorphism, refer to \cite{12}, Lemma 5.15, Remarks 5.16, 5.21.

\begin{theorem}
[\cite{5}, Theorem 1.5] Let $M$ be a complex manifold of dimension $n$, $E$ a compact analytic subset and $\{E_i\}_{i=1,...,s}$ all the $p$-dimensional irreducible components of $E$. If $T$ is a $\partial\overline{\partial}$-closed positive $(n-p, n-p)$-current on $M$ such that $supp T \subseteq E$, then there exist constants $c_i\geq 0$ such that $T- \Sigma_{i=1}^s c_i [E_i]$ is supported on the union of the irreducible components of $E$ of dimension larger than $p$.
\end{theorem}

\begin{theorem}
Let $X$ be a compact balanced space. If it has a desingularization $\widetilde{X}$, such that $i: H^{1, 1}_{BC}(\widetilde{X})\rightarrow H^{1,1}_{A}(\widetilde{X})$ is injective, then $X\in \mathscr{B}$.
\end{theorem}
\begin{proof}
  Set $dimX= n$. Suppose  $\pi: \widetilde{X}\rightarrow X $ is the desingularization. We need to prove that $\widetilde{X}$ is balanced. By Theorem 1.1, it suffices to prove that if $T$ is a positive $(1, 1)$-current on $\widetilde{X}$ which is the $(1, 1)$-component of a boundary, then $T= 0$.

Let $E\subseteq\widetilde{X}$ be the exceptional set of $\pi$, $\Omega$ the closed strictly positive $(n-1, n-1)$-form on $X$. Since $T$ is the $(1, 1)$-component of a boundary, we have $T(\pi^*\Omega) = 0$. On the other hand, we have
\begin{displaymath}
T(\pi^*\Omega)= \int_{\widetilde{X}} T\wedge \pi^*\Omega
\end{displaymath}
and $\pi^*\Omega$ is strictly positive on $\widetilde{X}-E$, so we obtain $suppT\subseteq E$.

By Theorem 4.4 for $p= n-1$, we obtain
\begin{displaymath}
 T= \sum_i c_i[E_i],
\end{displaymath}
where $c_i\geq0$ and  $E_i$ are the $(n-1)$-dimensional irreducible components of $E$. Hence $T$ is a closed current on $\widetilde{X}$. So $T$ is an element in the class $[T]\in H^{1, 1}_{BC}(\widetilde{X})$. Since $T$ is the $(1, 1)$-component of a boundary, $i([T])=0$. Beacause $i$ is injective, we know $[T]=0$. So, there is a real $0$-current $Q$ on $\widetilde{X}$, such that $T= i\partial\overline{\partial}Q$. Using a positive volume element of $\widetilde{X}$,  $Q$ can be identified with a distribution $\alpha$ with $i\partial\overline{\partial}\alpha\geq 0$, i.e. $\alpha$ is plurisubhamonic. Since $\widetilde{X}$ is compact, $\alpha$ is a constant, by maximum principle, hence $T= 0$.
\end{proof}

\begin{theorem}
 Let $X$ be a normal compact balanced space of dimension $n$ with the Betti number $b_{2n-1}(X)=0$. If $X$ has a desingularization  $\widetilde{X}$ satisfying:
 \begin{equation}\label{cond}
 \begin{aligned}
&\textup{any nonnegative  linear combination of  hypersurfaces contained in the exceptional}\\
 &\textup{set on $\widetilde{X}$ which is the $(1, 1)$-component of a boundary is $d$-exact,}
 \end{aligned}
\end{equation}
then $X\in \mathscr{B}$.
\end{theorem}
\begin{proof}
Suppose $T$ is a positive $(1, 1)$-current on $\widetilde{X}$ that is a $(1, 1)$-component of a boundary. As the proof in Theorem 4.5, we obtain
\begin{displaymath}
 T= \sum_i c_i[E_i]
\end{displaymath}
where $c_i\geq0$, $E_i$ are the $(n-1)$-dimensional irreducible components of $E$. By (\ref{cond}), $T= dQ$ for some current $Q$ on $\widetilde{X}$, so $\sum_i c_i [E_i]= [T]_{\widetilde{X}} = 0$ in $H_{2n-2}(\widetilde{X}, \mathbb{R})$. By Lemma 3.6, we get $c_i= 0$ for all $i$.
\end{proof}

\begin{proposition}
Let $X$ be a compact balanced space. If it has a desingularization  $\widetilde{X}$ whose exceptional set has codimension $\geq 2$
then $X\in \mathscr{B}$.\end{proposition}
\begin{proof}
Suppose $T$ is a positive $(1, 1)$-current on $\widetilde{X}$ which is the $(1, 1)$-component of a boundary, as the proof in Theorem 4.5, we obtain $suppT\subseteq E$. By Theorem 4.4 for $p= n-1$, we get $T= 0$ immediately.
\end{proof}

\begin{proposition}
If $f:X\rightarrow Y$ is a local isomorphism between reduced compact complex spaces and $Y$ is balanced, then $X$ is balanced. Moreover, if $f$ is an unramified covering map, then $X$ is balanced if and only if $Y$ is balanced.
\end{proposition}
\begin{proof}
Suppose $n= dimX= dimY$ and $\Omega_Y$ is a closed strictly positive $(n-1, n-1)$-form on $Y$. For all $x\in X$, there is an open neighbourhood $U$ of $x$ in $X$, an open neighbourhood $V$ of $f(x)$ in $Y$, such that $f\mid_U:U\rightarrow V$ is an isomorphism.  $(f^*\Omega_Y)|{U}=(f|_{U})^*(\Omega_Y|_{V})$ is obviously strictly positive and closed, so is $f^*\Omega_Y$. Therefore, $X$ is a balanced space.

If $f$ is an unramified covering map, then for every $y\in Y$, there exists an open neighbourhood $V\subseteq Y$ of $y$, and open subsets $U_1$, $...$, $U_d$ in $X$, which do not intersect with each other, such that $f^{-1}(V)= \cup_{i=1}^d {U_i}$ and the restriction $f|_{U_i}:U_i\rightarrow V$ is an isomorphism for $i=1,...,d$.
If $X$ is a balanced space and $\Omega_X$ is a closed strictly positive $(n-1, n-1)$-form  on $X$, we define a closed strictly positive form on $V$
\begin{displaymath}
\Omega_V:=\Sigma_{i=1}^d(f|_{U_i}^{-1})^*(\Omega_X|_{U_i})
\end{displaymath}
If $V$ and $V'$ are two open subsets in Y as above, and $V\cap V'\neq \phi$, we can easily check $\Omega_V= \Omega_{V'}$ on $V\cap V'$. Hence we can construct a global closed strictly positive $(n-1, n-1)$-form $\Omega_Y$ on $Y$, such that $\Omega_Y|_V= \Omega_V$.
\end{proof}

\section{Appendix}

For the reader's convenience, we collect here some terminology and results needed in this paper. We refer to \cite{21} and \cite{23}. The following notation ``strictly positive $(p, p)$-form" is called positive in \cite{16}, transverse positive in \cite{3}, an element of the interior of the cone of weakly positive $(p, p)$-forms in \cite{13} and \cite{18}.

Let $V$ be a complex $n$-dimensional vector space, $I= (i_1, ..., i_p)$ be a multi-index. For $a_{i_1}, ..., a_{i_p}\in \bigwedge^{1, 0}V$, define $a^I:=a_{i_1}\wedge ... \wedge a_{i_p}$. Denote the set of real $(p, p)$-forms by $\bigwedge_\mathbb{R}^{p, p}V$.

\begin{definition}
If ${e_1, ..., e_n}$ is a basis of $\bigwedge^{1, 0}V$, an $(n, n)$-form $v$ is called a strictly positive volume form if $v= c ie_1\wedge\overline{e_1}\wedge ... \wedge ie_n\wedge\overline{e_n}$, with $c> 0$.
\end{definition}

\begin{definition}
An element $\alpha$ of $\bigwedge_\mathbb{R}^{p, p}V$ is called strictly positive if $\alpha\wedge i^{(n-p)^2} a^I\wedge \overline{a^I}$  is a strictly positive volume form for all $a_1, ..., a_{n-p}$ linearly independent in $\bigwedge^{1, 0}V$.
\end{definition}

\end{document}